\pdfoutput=1 
\documentclass[11pt]{amsart}
\usepackage{epsfig,overpic,comment,amssymb}
\usepackage[usenames,dvipsnames,svgnames,table]{xcolor}
\usepackage[pagebackref,linktocpage=true,colorlinks=true,linkcolor=Blue,citecolor=BrickRed,urlcolor=RoyalBlue]{hyperref}
\usepackage[msc-links,abbrev]{amsrefs} 

\newtheorem{thm}{Theorem}[section]
\newtheorem{lem}[thm]{Lemma}
\newtheorem{cor}[thm]{Corollary}
\newtheorem{prop}[thm]{Proposition}

\theoremstyle{definition}

\newtheorem{note}[thm]{Note}

\newcommand{\R}{\mathbf{R}}

\newcommand{\ol}{\overline}

\renewcommand{\d}{\partial}

\renewcommand{\S}{\mathbf{S}}
\renewcommand{\l}{\langle}
\renewcommand{\r}{\rangle}

\renewcommand{\tilde}{\widetilde}
\renewcommand{\angle}{\measuredangle}

\DeclareMathOperator{\dist}{dist}
\DeclareMathOperator{\inte}{int}

\DeclareMathOperator{\diam}{diam}
\DeclareMathOperator{\length}{length}

\title[Pseudo-edge unfoldings of convex polyhedra]{Pseudo-edge unfoldings of convex polyhedra}

\author{Nicholas Barvinok}
\address{School of Mathematics, Georgia Institute of Technology,
Atlanta, GA 30332, USA}
\email{nbarvinok3@gatech.edu}
\urladdr{www.math.gatech.edu/~nbarvinok3/}

\author{Mohammad Ghomi}
\address{School of Mathematics, Georgia Institute of Technology,
Atlanta, GA 30332, USA}
\email{ghomi@math.gatech.edu}
\urladdr{www.math.gatech.edu/~ghomi}

\date{\today \,(Last Typeset)}
\thanks{Research of the second named author was supported in part by NSF Grant DMS--1308777.}

\begin{document}
\maketitle

\begin{abstract}
A pseudo-edge graph of a convex polyhedron $K$ is a $3$-connected embedded graph in $K$ whose vertices coincide with those of $K$, whose edges are distance minimizing geodesics, and whose faces are convex. We construct a convex polyhedron $K$ in Euclidean 3-space with a pseudo-edge graph with respect to which $K$ is not unfoldable. 
The proof is based on a result of Pogorelov on convex caps with prescribed curvature, and an unfoldability obstruction for almost flat convex caps due to Tarasov.
Our example, which has $340$ vertices, significantly simplifies an earlier construction by Tarasov, and confirms that D\"{u}rer's conjecture does not hold for pseudo-edge unfoldings.
\end{abstract}

\section{Introduction}
By a \emph{convex polyhedron} in this work we mean the boundary of  the convex hull of finitely many points in Euclidean space $\R^3$ which do not all lie in a plane.
 A well-known conjecture \cite{do:book}, attributed to the Renaissance painter Albrecht D\"{urer} \cite{durer}, states that every convex polyhedron $K$ is \emph{unfoldable}, i.e., it may be cut along some spanning tree of its edges and isometrically embedded into the plane $\R^2$. Here we study a generalization of this problem to \emph{pseudo-edges} of $K$, i.e.,  distance minimizing geodesic segments in $K$ connecting  pairs of its vertices (see Figure \ref{fig:cube} for an example of a pseudo-edge which is not an actual edge). A \emph{pseudo-edge graph} $E$ of $K$ is a 3-connected embedded graph composed of pseudo-edges of $K$, with the same vertices as those of $K$, and with faces which are \emph{convex} in $K$, i.e., the interior angles of each face of $E$ are less than $\pi$. Cutting $K$ along any spanning tree $T$ of $E$ yields a simply connected compact surface $K_T$ which admits an isometric immersion  or \emph{unfolding} $u_T\colon K_T\to\R^2$. If $u_T$ is one-to-one for some  $T$, then we say that $K$ is \emph{unfoldable} with respect to $E$. The main result of this paper is as follows:

\begin{thm}\label{thm:main}
There exists a convex polyhedron $K$ with $340$ vertices and a pseudo-edge graph with respect to which $K$ is not  unfoldable.
\end{thm}

Thus one may say that D\"{u}rer's conjecture does not hold in a purely intrinsic sense, since it is not possible to distinguish a pseudo-edge from an actual edge by means of local measurements within $K$. On the other hand, by Alexandrov's isometric embedding theorem \cite{alexandrov:polyhedra}, any 
convex polyhedron is determined up to a rigid motion by its intrinsic metric. So edges do indeed exist intrinsically, although Alexandrov's proof is not constructive and does not specify their location. A more constructive approach has been studied by Bobenko and Izmestiev \cite{bobenko-izmestiev} but that too does not yield a simple characterization for the edges. In short, the edges of convex polyhedra are not well understood from the  point of view of isometric embeddings, and, in light of the above theorem, it would now be even more remarkable if the conjecture holds.

\begin{figure}[h]
\centering
\begin{overpic}[height=1.2in]{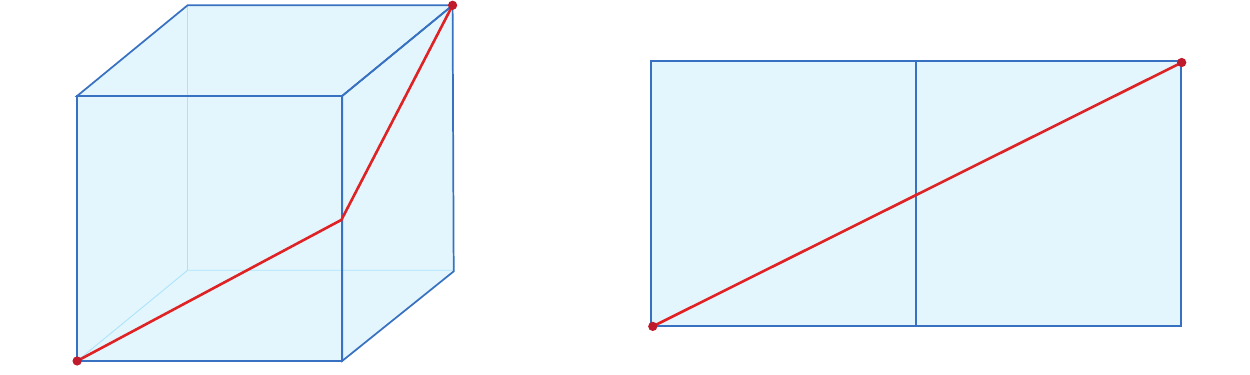}
\end{overpic}
\caption{}\label{fig:cube}
\end{figure}

Theorem \ref{thm:main}, for a polyhedron with over 19,000 vertices, was first announced in 2008  in a highly original and hitherto unpublished manuscript by Alexey Tarasov \cite{tarasov}. Although we do not understand all the details in that construction, since it is very complex, we can confirm that Tarasov's key ideas were correct, and utilize these in this work. These notions, which will be described below,  include the obstruction for unfoldability of almost flat convex caps in Section \ref{sec:pseudo}, and the double spiral configuration in Section \ref{sec:gadget}. 

The polyhedron $K$ in Theorem \ref{thm:main} is obtained by arranging $4$ congruent almost flat  convex caps over the faces of a regular tetrahedron.  These caps have $84$ interior vertices each with prescribed curvature and projection. They are constructed via a result of Pogorelov on convex caps with prescribed curvature as we describe in Section \ref{sec:caps}. In Section \ref{sec:pseudo} we study the pseudo-edges induced on a convex cap $C$ by the edge graph $G$ of  convex subdivisions of the polygon at the base of $C$. Then in Section \ref{sec:cut} we describe a necessary condition, due to Tarasov, for unfoldability of $C$ in terms of spanning forests of $G$. Next in Section \ref{sec:gadget} we construct a convex subdivision of an equilateral triangle which satisfies Tarasov's criterion. Consequently, sufficiently flat convex caps constructed over this subdivision fail to be unfoldable with respect to the induced pseudo-edge graph.  Finally in Section \ref{sec:proof} we assemble $4$ such caps to construct  $K$.

Our construction differs from Tarasov's in the following  respects. First, we use only two spiral paths, as opposed to three, in the subdivision of the equilateral triangle mentioned above. Second, our double spiral configuration in the center of the triangle uses far fewer vertices and thus is more transparent. Third, the corresponding convex caps we construct have planar boundaries due to our use of Pogorelov's theorem mentioned above, whereas Tarasov applies instead a related result of Alexandrov for unbounded polyhedra which does not yield precise control at the boundary of the cap. Fourth, since our caps have planar boundaries, we require only $4$ copies of them to assemble our polyhedron, whereas Tarasov employs many more in a complex configuration.  

The edge unfolding problem for convex polyhedra was first explicitly formulated by Shephard \cite{shephard} in 1975, and since then has been advertised in several sources, e.g., \cite{pak:book,ziegler:book, cfg, do:book, orourke:book, ghomi:notices}. The conjecture that the answer  is yes, i.e., all convex polyhedra are unfoldable, appears to be first stated by Gr\"unbaum  \cite{grunbaum} in 1991. The earliest known examples of unfoldings of convex polyhedra were drawn by D\"{u}rer \cite{durer}  in 1525, all of which were nonoverlapping. Hence the unfolding problem or conjecture are often associated with his name.  For more background, references, and a positive recent result see \cite{ghomi:durer} where it is shown that every convex polyhedron becomes unfoldable after an affine transformation. See also O'Rourke \cite{orourke2016,orourke2017} for other recent positive results  concerning unfoldability of certain convex caps. As far as we know, Theorem \ref{thm:main}  is the first hard evidence against D\"{u}rer's conjecture.

\section{Convex Caps with Prescribed Boundary and Curvature}\label{sec:caps}
 A (polyhedral)  \emph{convex cap} $C\subset\R^3$ is a topological disk which lies on a convex polyhedron and whose boundary $\d C$ lies in a plane $H$, while its interior $C\setminus\d C$ is disjoint from $H$.
 The \emph{normal cone} $N_p(C)$ of $C$ at an interior point $p$ is the convex cone generated by all outward normal vectors to support planes of $C$ at $p$. The \emph{unit normal cone} $\ol N_p(C)$ is the collection of unit vectors in $N_p(C)$. The \emph{curvature} of $C$ at $p$ is defined as 
 $$
 \kappa(p)=\kappa_C(p):=\sigma(\ol N_p(C)),
 $$
  where $\sigma$ denotes the area measure in the unit sphere $\S^2$.      Let $\pi\colon\R^3\to\R^2$ denote the projection into the first two coordinates. A set $X\subset\R^3$ is a \emph{terrain} over $\R^2$ provided that $\pi$ is one-to-one on $X$, and $X\subset\R^2\times [0,\infty)$.
A \emph{convex polygon} $P$ is the convex hull of finitely many points in $\R^2$ which do not all lie on a line. We say that a convex cap $C$ is  \emph{over} $P$ provided that $C$ is a terrain over $\R^2$ and $\d C=\d P$.   We need the following result of Pogorelov \cite[Lem. 1, p. 65]{pogorelov:book}, see also Pak's lecture notes \cite[Thm 35.7]{pak:book}. 

\begin{lem}[Pogorelov \cite{pogorelov:book}]\label{lem:C}
Let $P$ be a convex polygon,  $p_i$, $i=1,\dots, n$,  be  points in the interior of $P$,  and $\beta_i>0$ with $\sum_{i}\beta_i<2\pi$. Then there exists a unique convex cap $C$ over $P$ with  interior vertices $v_i$ such that  $\pi(v_i)=p_i$,  and $\kappa(v_i)=\beta_i$.\qed
\end{lem}

A \emph{convex subdivision} of a convex polygon $P$ is a subdivision of $P$  into finitely many convex polygons each of whose vertices either lies in the  interior of $P$ or coincides with a vertex of $P$ (we assume that the interior angles of $P$ at all its vertices are less than $\pi$). If $G$ is the (edge) graph of a convex subdivision of $P$, then by an \emph{interior vertex} $p_i$ of $G$ we mean a vertex of $G$ which lies in the interior of $P$. We assume that the angles of incident edges of $G$ at $p_i$ are all less than $\pi$. We say that $G$ is \emph{weighted} if to each of its interior vertices  there is associated a number $\alpha_i>0$ with $\sum_i\alpha_i=1$. 
Let the \emph{total curvature} $\kappa(C)$ of a convex cap $C$ be the sum of the curvatures of its interior vertices.
Lemma \ref{lem:C} immediately yields:

\begin{cor}\label{cor:C}
Let $P$ be a  convex polygon, and $G$ be the weighted graph of a convex subdivision of $P$, with interior vertices $p_i$ and weights $\alpha_i$. Then for any $0<\beta< 2\pi$  there exists a convex cap $C_\beta$ over $P$ with interior vertices $v_i$ such that $\pi(v_i)=p_i$ and $\kappa(v_i)=\beta_i:=\alpha_i\beta$. In particular $\kappa(C_\beta)=\beta$.\qed
\end{cor}

\section{Pseudo-Edge Unfoldings of Almost Flat Convex Caps}\label{sec:pseudo}

In this section we fix  $P$, $G$,  and $\alpha_i$ to be as in Corollary \ref{cor:C}, and aim to study the corresponding convex caps $C_\beta$ for small $\beta$. In particular we will show that $G$ gives rise to a unique pseudo-edge graph $\ol G$ of $C_\beta$ (Proposition \ref{prop:E}) and study the corresponding unfoldings of $C_\beta$ in relation to $P$ (Proposition \ref{prop:converge}). 

\subsection{The induced pseudo-edge graph of $C_\beta$}
First we check that as $\beta\to 0$, $C_\beta\to P$. More precisely, if
$d_\beta$ denotes the intrinsic distance in $C_\beta$, then we have:

\begin{lem}\label{lem:BetaToZero}
As $\beta\to 0$, $d_\beta(x,y)\to|\pi(x)-\pi(y)|$, for all $x$, $y\in C_\beta$.
\end{lem}
\begin{proof}
As $\beta\to 0$, the maximum height of $C_\beta$ goes to zero. If not,
there exists a sequence $\beta_k\to 0$ such that the maximum height of $C_k:=C_{\beta_k}$ is bounded below by $h>0$. So, after refining the subsequence $C_k$ further, we may assume that for some $i$, the height of the vertex $v^k_i$ of $C_k$ which projects onto $p_i$  is bounded below by $h$. Let $o$ be the point of height $h$ above $p_i$, and $C'$ be the convex cap formed by line segments connecting $o$ to points of $\partial P$. 
Since $C'$ lies below $C_k$, every support plane of $C'$ at an interior vertex is parallel to a support plane of $C_k$ at an interior vertex. So $\kappa(C_k)\geq \kappa(C')>0$, which is the desired contradiction since $\kappa(C_k)=\beta_k\to 0$. 

Now let $L$ be the line segment connecting $\pi(x)$, $\pi(y)$, and $\ol L$ be the corresponding curve in $C_\beta$ connecting $x$, $y$, such that $\pi(\ol L)=L$. Then $d_\beta(x,y)\leq\length(\ol L)$. But $\ol L$ is the graph of a convex function over $L$ which converges to $0$. Thus $\length(\ol L)\to\length(L)=|\pi(x)-\pi(y)|$. So the limit of $d_\beta(x,y)$ is not bigger than $|\pi(x)-\pi(y)|$. On the other hand $d_\beta(x,y)\geq |x-y|\geq |\pi(x)-\pi(y)|$. So $d_\beta(x,y)\to|\pi(x)-\pi(y)|$.
\end{proof}

A \emph{polyhedral disk} $D$ is a topological disk composed of a finite number of convex polygons identified along their edges. We say that $D$ is \emph{flat} if the total angle at each of its interior vertices is $2\pi$. An \emph{isometric immersion} $f\colon D\to\R^2$ is a locally one-to-one continuous map which preserves distances between points on each face of $D$. If  $f$ is one-to-one everywhere, then we say that it is an \emph{isometric embedding}.

\begin{lem}\label{lem:develop}
Let $D$ be a flat polyhedral disk. Suppose that the total angle at each of the boundary vertices of $D$ is less than $2\pi$. Then there exists an isometric immersion $D\to\R^2$.
\end{lem}
 \begin{proof}
 The angle condition along $\d D$ ensures that each point of $D$ has a neighborhood which may be isometrically embedded into $\R^2$. Since $D$ is simply connected, a family of these local maps may be joined to produce the desired global map, e.g., see the proof of \cite[Lem. 2.2]{ghomi:verticesA} for further details.
 \end{proof}

 By a  \emph{geodesic} in $C_\beta$ we mean the image of a continuous map $\gamma\colon[a,b]\to C_\beta$ such that $\length[\gamma]=d_\beta(\gamma(a),\gamma(b))$. For any set $X\subset\R^2$, and $r>0$, $U_r(X)$ denotes the (open) set of points in $\R^2$ which are within a distance $< r$ of $X$. Further we set
\begin{equation}\label{eq:delta}
 \delta:=\min \dist(p_ip_j, p_k)
 \end{equation}
where $p_i$, $p_j$ range over all pairs of adjacent vertices of $G$, so $p_ip_j$ indicates an edge of $G$ (viewed as a line segment in $\R^2$), and $p_k$ ranges over vertices different from $p_i$ and $p_j$. Hence $\delta>0$. By \emph{sufficiently small}, or simply \emph{small}, throughout this work we mean all nonzero values smaller than some positive constant. More explicitly, we say that some property holds for $\beta$ sufficiently small, provided that there exists a constant $\beta_0>0$ such that the property holds for all $0<\beta\leq\beta_0$.
 
\begin{lem}\label{lem:E}
 If $\beta$ is sufficiently small, then to each edge $e$ of $G$ there corresponds a unique  geodesic $\ol e$ of $C_\beta$ whose end points project into the endpoints of $e$, and $\pi(\ol{e})\subset U_\delta(e)$. 
 \end{lem}
 \begin{proof}
 Let $x$, $y\in C_\beta$ be points which project into the end points of $e$, and $\Gamma$ be a geodesic in $C_\beta$ connecting $x$ and $y$. By Lemma \ref{lem:BetaToZero}, $\length(\Gamma)\to |\pi(x)-\pi(y)|$. Thus, for $0<\beta\leq\beta_0(e)$,  $\pi(\Gamma)\subset U_\delta(e)$. We claim that $\Gamma$ is unique. To this end suppose, towards a contradiction, that there exists another geodesic $\Gamma'$ in $C_\beta$ connecting $x$ and $y$, which is different from $\Gamma$. Then again we have $\pi(\Gamma')\subset U_\delta(e)$, since by definition our geodesics are length minimizing, and so $\length(\Gamma')=d_\beta(x,y)=\length(\Gamma)$. Let $V\subset C_\beta$ be the region with $\pi(V)=U_\delta(e)$. Then $\Gamma$, $\Gamma'\subset V$. Since $\Gamma\neq\Gamma'$ and $V$ is simply connected, there exists a simply connected domain $D\subset V$ bounded by a pair of subsegments $\Gamma_0$ and $\Gamma_0'$ of $\Gamma$ and $\Gamma'$ respectively. Note that, by our choice of $\delta$, see \eqref{eq:delta}, $V$ does not contain any vertices of $C_\beta$ other than $x$ and $y$. Thus $D$ does not contain any vertices in its interior. So $D$ admits an isometric immersion $f\colon D\to\R^2$ by Lemma \ref{lem:develop}.  But, since isometries preserve geodesics, $f$ maps $\Gamma_0$ and $\Gamma_0'$ to straight line segments (which have the same end points). Hence $f(\Gamma_0)=f(\Gamma_0')$. In particular $f$ is not locally injective at the points of $\d D$ where $\Gamma_0$ and $\Gamma_0'$ meet, which is the desired contradiction. So $\Gamma$ is indeed unique. Finally, setting $0<\beta\leq \min \beta_0(e)$, as $e$ ranges over all edges of $G$, completes the proof.
 \end{proof}

 A \emph{convex polygon} $X$ in $C_\beta$ is a region bounded by a simple closed curve composed of a finite number of  geodesics meeting at angles which are less than $\pi$ with respect to the interior of $X$. A \emph{convex subdivision} of $C_\beta$ is a subdivision into finitely many convex polygons whose interiors contain no vertices of $C_\beta$, and whose vertices are vertices of $C_\beta$. A \emph{pseudo-edge graph} of $C_\beta$ is the edge graph of a convex subdivision.  By Lemma \ref{lem:E}, for $\beta$ sufficiently small,  there exists a unique pseudo-edge graph $\ol G$ of $C_\beta$ such that $\pi(\ol G)\subset U_\delta(G)$. Let $G^T$ be the (canonical) triangulation of $G$ given by connecting the center of mass of each nontriangular face of $G$ to its vertices. Again by Lemma \ref{lem:E}, there exists a unique triangulation $\ol G^T$ of $\ol G$ such that $\pi(\ol G^T)\subset U_\delta(G^T)$ and the vertices of $\ol G^T$ project onto the vertices of $G^T$. For any triangle $\Delta$ of $G^T$, let $\ol \Delta$ be the  triangle of $\ol G^T$ whose vertices project onto the vertices of $\Delta$. We will refer to $\Delta$ simply as a \emph{triangle} of $G$, and call $\ol\Delta$ the corresponding triangle of $\ol G$. Note that, by Lemma \ref{lem:develop}, there exists an isometric embedding $u_\Delta\colon \ol\Delta\to\R^2$ for each triangle $\Delta$ of $G$.

\begin{prop}\label{prop:E}
For $\beta$ sufficiently small,  there exists a canonical homeomorphism $f\colon P\to C_\beta$ such that (i) $f$ is the identity on $\d P$, (ii) $f(G)=\ol G$, and (iii) $u_\Delta\circ f$ is an affine map on each triangle $\Delta$ of $G$. Furthermore, $f$ converges to the identity map on $P$ as $\beta\to 0$.
\end{prop}
\begin{proof}
 For any vertex $p$ of $G^T$ let $f(p):=\pi^{-1}(p)\cap C_\beta$.  We define a mapping $g_\Delta\colon\Delta\to\Delta':=u_\Delta(\Delta)$ as follows.
For any $x\in\Delta$, let $(x_1,x_2,x_3)$ be the barycentric coordinates of $x$ with respect to the vertices $v_1$, $v_2$, $v_3$ of $\Delta$. Let $g_\Delta(x)$ be the point of $\Delta'$ whose barycentric coordinates with respect to the vertices $v_i':=u_\Delta\circ\pi^{-1}(v_i)$ of $\Delta'$ are $(x_1,x_2,x_3)$. Finally  set $f(x):=u_\Delta^{-1}(g_\Delta(x))$, where $\Delta$ is a triangle of $G^T$ which contains $x$. Since $C_\beta$ converges to $P$, as $\beta\to 0$, it follows that $f$ converges to the identity map on $P$.
\end{proof}

\subsection{Cut forests and unfoldings of $C_\beta$}\label{subsec:cutforest}
A \emph{tree} is a connected graph without cycles. A subgraph $F$ of $G$ is called a \emph{cut forest} if 
(i) $F$ is a collection of disjoint trees which contain all the vertices of $G$ in the interior of $P$,
and (ii) each tree of $F$ contains exactly one vertex of $\d P$; see the middle diagram in Figure \ref{fig:forest}. 
\begin{figure}[h]
\centering
\begin{overpic}[height=1.4in]{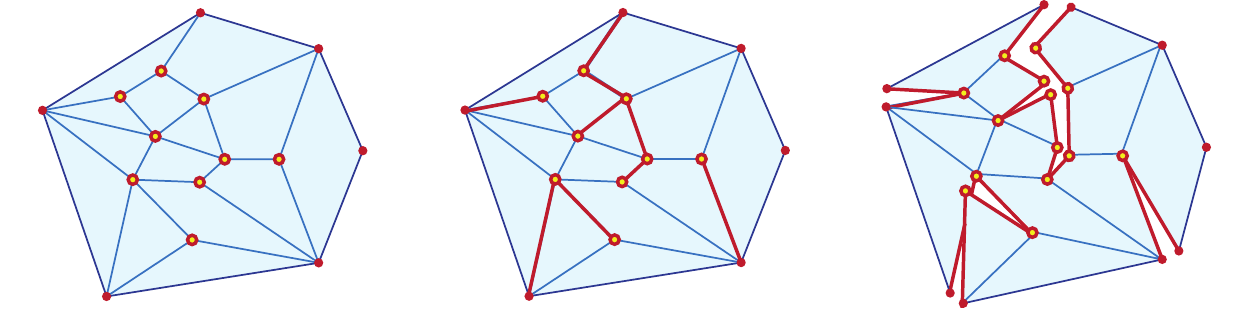}
\put(4.5,2.75){\Small$P$}
\put(17.5,14){\Small$G$}
\put(51.5,14.25){\Small$F$}
\put(69.5,3){\Small$C'_{\beta,F}$}
\end{overpic}
\caption{}\label{fig:forest}
\end{figure}
By Lemma \ref{lem:E}, to each cut forest $F$ of $G$ there corresponds a unique cut forest $\ol F$ of $\ol G$, assuming $\beta$ is small. Let $C_{\beta,F}$ be the surface obtained from $C_\beta$ by \emph{cutting} it along $F$, i.e., take the disjoint collection of the faces of $\ol G$ and glue them together pairwise along all their common edges which do not belong to $F$. 

By Lemma \ref{lem:develop} there exists an isometric immersion, or  \emph{unfolding map} 
$
u\colon C_{\beta,F}\to\R^2.
$
 We assume that $u$ fixes a designated edge $e_0$ of  $\d C_\beta$, and locally maps $C_{\beta,F}$ to the same side of $e_0$ where $P$ lies.   Let
  $
 \textup{cov}\colon C_{\beta,F}\to C_\beta
 $ 
 be the natural covering map which sends each face of $C_{\beta,F}$ to the corresponding face of $C_\beta$, 
and 
$
f\colon P\to C_\beta
$ 
be the homeomorphism given by Proposition \ref{prop:E}. Then
\begin{equation}\label{eq:phi}
\psi:=u\circ \textup{cov}^{-1}\circ f,
\end{equation}
is a multivalued mapping
$
P\to C'_{\beta,F}:=u(C_{\beta,F}).
$
Since $f$ converges to the identity map on $P$, it follows that $\psi$ also converges to the identity  on $P$ as $\beta\to 0$. 
Note that $\psi$ is single valued on $P\setminus F$ and is one-to-one on the interior of each face $\Phi$ of $G$. But $\psi$ is doubly valued for points in the interior of each edge of $F$ (where a cut occurs). Furthermore, if $x$ is a vertex of $F$, then the cardinality of $\psi(x)$ is equal to the degree of $x$ in $F$ (see Figure \ref{fig:forest}).
For any point $x\in P$, we set 
$$
x':=\psi(x).
$$
 So $x'$ in general indicates a set of points, not a single point. Whenever we state that $x'$ satisfies some property, we mean that \emph{each element} of $x'$ satisfies that property.

Note that $\psi$ induces a natural single-valued map $\psi_\Phi$ on each face $\Phi$ of $G$ as follows.
Let $\Phi'$ indicate the face of $C'_{\beta, F}$ corresponding to $\Phi$,
i.e., the closure of $\psi(\inte(\Phi))$. Then we obtain  a homeomorphism $\psi_\Phi$ between $\Phi$ and $\Phi'$, by setting $\psi_\Phi:=\psi$ on $\inte(\Phi)$ and extending the map continuously to the boundary of $\Phi$. For any $x\in \Phi$, let 
$$
x'_\Phi:=\psi_\Phi(x)
$$
 denote the corresponding (single) point of $\Phi'$. Since $\psi$ converges to the identity on $P$, it follows that, as $\beta\to 0$, $x'_\Phi\to x$ for all faces $\Phi$ of $G$ which contain $x$. In other words, $\psi_\Phi$ converges pointwise to the identity map $\textup{id}_\Phi$ on $\Phi$. Thus, as $\psi_\Phi$ is continuous and $\Phi$ is compact, we conclude that
 \begin{equation}\label{eq:phiPhi}
 \psi_\Phi\to \textup{id}_\Phi\quad\text{uniformly as}\quad \beta\to 0
\end{equation}
for all faces $\Phi$ of $G$. Next note that for every point $x$ in the interior of a triangle $\Delta$ of $\Phi$ we have
$$
\psi_\Phi(x)=\psi(x)=u\circ\textup{cov}^{-1}\circ f(x)=u\circ f(x)=u_\Delta\circ f(x).
$$
Also recall that, by Proposition \ref{prop:E}, $u_\Delta\circ f$ is an affine map.  Thus
 \begin{equation}\label{eq:phiPhi2}
\psi_\Phi \; \text{is an affine map on each triangle $\Delta\subset\Phi$}.
\end{equation}
These properties of $\psi_\Phi$ yield that:

\begin{prop}\label{prop:converge}
For every $\epsilon>0$, there exists $\beta_0(\epsilon)>0$ such that for all $0<\beta\leq \beta_0(\epsilon)$, $x\in P$, and faces $\Phi$ of $G$ which contain $x$,
 \begin{equation}\label{eq:xx'}
|x-x'_\Phi|\leq\epsilon.
 \end{equation} 
Furthermore, for any pair of points $x$, $y$ which both lie in the same triangle $\Delta\subset\Phi$, 
 \begin{equation}\label{eq:x-y}
|(x-y)-(x'_\Phi-y'_\Phi)|\leq\epsilon|x-y|.
 \end{equation} 
\end{prop}
\begin{proof}
The  inequality \eqref{eq:xx'} follows immediately from \eqref{eq:phiPhi}. 
To see \eqref{eq:x-y}, note that by  \eqref{eq:phiPhi2} $\psi_\Phi$ may be extended to an affine mapping from $\R^2$ to $\R^2$, and thus  be written as $\ell+c$ for a linear transformation $\ell\colon\R^2\to \R^2$ and some fixed vector $c\in\R^2$. Thus
$$
x'_\Phi-y'_\Phi=\psi_\Phi(x)-\psi_\Phi(y)=\ell(x-y).
$$
If $x=y$, \eqref{eq:x-y} already holds. Otherwise, we may set $z:=(x-y)/|x-y|$ and divide the left hand side  of \eqref{eq:x-y} by $|x-y|$ to obtain
\begin{eqnarray*}
\frac{|(x-y)-(x'_\Phi-y'_\Phi)|}{|x-y|}=\left|\frac{x-y}{|x-y|}-\ell\left(\frac{x-y}{|x-y|}\right)\right|=|z-\ell(z)|.
\end{eqnarray*}
Finally note that since $\psi_\Phi\to\textup{id}_\Phi$, $\ell$ must converge to the identity map on any compact subset of $\R^2$. In particular, we may choose $\beta_0(\epsilon)$ so small that $|z-\ell(z)|\leq\epsilon$ for all unit vectors $z\in\S^1$, which completes the proof.
\end{proof}

\section{Tarasov's Monotonicity Condition}\label{sec:cut}
As in the last section, let $P$ be a convex polygon, $G$ be the graph of a fixed convex subdivision of $P$ with weights $\alpha_i$, and $C_\beta$ be the corresponding convex cap over $P$ given by Corollary \ref{cor:C}. 
Here we describe Tarasov's obstruction for unfoldability of $C_\beta$ with respect to the induced pseudo-edge graph $\ol G$ given by Proposition \ref{prop:E}.

By an \emph{edge} of $G$ we mean the line segment connecting a pair of adjacent vertices of $G$, and a \emph{point} of $G$ is any point of an edge of $G$. We say a pair of  points are adjacent if they belong to the same edge. A \emph{path} $\Gamma$ in $G$ is a sequence of points, each adjacent to the next, all of which are vertices except possibly  the initial point. We say that $\Gamma$ is \emph{simple} if the sequence of line segments determined by its consecutive points forms a non-self-intersecting curve.  If $F$ is a cut forest of $G$, then
each point $p$ of $F$  may be joined to $\d P$ with a unique simple path $\Gamma_p$ in $F$, which we call the \emph{ancestral path} of $p$. This induces a  partial ordering on points of $F$ as follows: we write $y\succeq x$, for $x$, $y\in F$ and say that  $y$ is a \emph{descendant} of $x$ or $x$ is an \emph{ancestor} of $y$,  if $x\in\Gamma_y$. In particular note that $x\succeq x$. If $y\succeq x$ and $x\neq y$ then we  say $y$ is a \emph{strict descendant} of $x$, or $x$ is a strict ancestor of $y$, and write $x\succ y$. Furthermore, we adopt the following convention: for any $x\in F$, we write $i\succeq x$ provided that $p_i\succeq x$, where  $p_i$ denote the vertices of $G$.

For any point $x\in G$, we define the \emph{center of rotation} of  $x$ as the center of mass of its descendant vertices with respect to the weights $\alpha_i$:
  $$
 c_x:=\alpha_x^{-1}\sum_{i\succeq x}\alpha_ip_i, \quad\text{where}\quad \alpha_x:=\sum_{i\succeq x}\alpha_i.
 $$
 Roughly speaking, $c_x$ is the limit, as $\beta\to 0$, of the pivot point about which $x$ rotates (in different directions) to generate (the elements of) $x'$ defined in the last section (see Note \ref{note:pivot}).

Every interior vertex $p_i$ of $G$ has a unique adjacent vertex $p_i^*$ in $F$ which is its \emph{parent} or first strict ancestor which is a vertex. We also refer to $p_i$ as a \emph{child} of $p_i^*$. A  cut forest $F$ of $G$ is called \emph{monotone} (in the sense of Tarasov),  if for every interior vertex $p_i$ of $G$ we have
\begin{equation}\label{eq:cxe}
\l p_i^*-p_i,p_i-c_{i}\r\geq 0,\quad\text{where}\quad c_i:=c_{p_i}.
\end{equation}
In other words, $p_i^*$ must lie in the set $\l x-p_i,p_i-c_{i}\r\geq 0$, which forms a half-plane when $c_i\neq p_i$;
see figure \ref{fig:monotone}. 
\begin{figure}[h]
\centering
\begin{overpic}[height=1in]{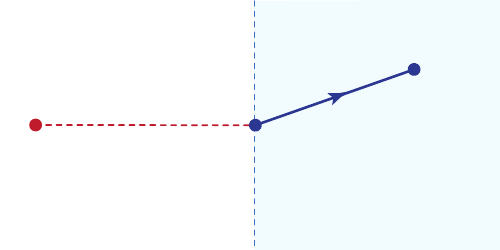}
\put(50,20){\small $p_i$}
\put(85,36){\small $p_i^*$}
\put(0,23){\small $c_i$}
\end{overpic}
\caption{}\label{fig:monotone}
\end{figure}
So, if $c_i\neq p_i$ and $0\leq\angle p_i^*p_i c_i\leq\pi$ denotes the angle between the vectors $p_i^*-p_i$ and $c_i-p_i$, then we have
\begin{equation}\label{eq:cxe2}
\angle p_i^*p_i c_i \geq\pi/2.
\end{equation}
In particular, every parent must be further away from the center of rotation of its child than the child is: 
\begin{equation}\label{eq:cxe3}
|p_i^*-c_i|>|p_i-c_i|.
\end{equation}
 Hence the term ``monotone". Note that \eqref{eq:cxe2} is equivalent to \eqref{eq:cxe} whenever $c_i\neq p_i$; however, \eqref{eq:cxe3} is a strictly weaker notion. Below we will primarily use the form \eqref{eq:cxe2} of the monotonicity condition. Some other notions of monotonicity have been studied recently by O'Rourke \cite{orourke2016,orourke2017}, and Lubiw and O'Rourke \cite{lubiw-orourke2017} for cut forests of  convex polyhedral disks. One of these notions, called radial monotonicity, will be invoked  in 
 Section \ref{subsec:properties} below as it is somewhat related to \eqref{eq:cxe2}. See \cite{ghomi:durer} for yet another monotonicity notion in the context of unfoldings.
 
 Recall that $C'_{\beta, F}$ is  the image of the unfolding map $u\colon C_{\beta, F}\to\R^2$. We say that $C'_{\beta, F}$ is  \emph{simple},  if and only if $u$ is injective. If $C'_{\beta, F}$ is simple for some cut forest $F$ of $G$, we say that $C_\beta$ is \emph{unfoldable} with respect to $\ol G$. If $G$ admits no monotone cut forests, then we say that $G$ is \emph{non-monotone}. The rest of this section is devoted to establishing the following result which parallels \cite[Thm. 1]{tarasov}. Recall that by sufficiently small throughout the paper,  as we stated in Section \ref{sec:pseudo}, we mean for all values smaller than some constant.

\begin{thm}\label{thm:F}
If $G$ is non-monotone, then $C_\beta$ is not unfoldable with respect to $\ol G$, for $\beta$ sufficiently small. 
 \end{thm}

We prove the above theorem via the same  general approach indicated in \cite{tarasov}, although we correct a number of errors or ambiguities, provide more details, and make many simplifications. 
Fix a cut forest $F$ of $G$.
 If $x\in F$ is  not a vertex,  $x'$ consists of precisely two elements: $x'_R$ and $x'_L$ defined as follows. Orient the edge $e$ of $F$ containing $x$ from the child to the parent vertex. Then we can distinguish the faces $\Phi_R$, $\Phi_L$ of $G$ which lie to the right and left of $e$ respectively. We set 
 $$
 x'_R:=x'_{\Phi_R}, \quad\quad\text{and}\quad\quad x'_L:=x'_{\Phi_L}.
 $$
 Let  $J$ be the $\pi/2$-clockwise rotation about the origin of $\R^2$, and set
\begin{equation}\label{eq:tildec}
\tilde c_x:=x+\frac{J(x'_R-x'_L)}{\beta_x},\quad\quad\text{where}\quad\quad \beta_x:=\alpha_x\beta=\sum_{i\succeq x}\beta_i.
\end{equation}
The next observation parallels \cite[Lem. 1]{tarasov}.
 
\begin{lem}\label{lem:T1}
For every $x\in F$,
$
\tilde{c}_x\to c_x,
$
as $\beta\to 0$.
\end{lem}

\begin{proof}
Let $F_x:=\{y\in F\mid y\succeq x\}$,
and $\Gamma$ be a polygonal Jordan curve in $P$  which encloses $F_x$ and intersects $F$ only at $x$; see the left diagram in Figure \ref{fig:leaf}.  Then $\Gamma':=\psi(\Gamma)$ is a polygonal path connecting $x'_R$ and $x'_L$. Let $U$ be the region bounded by $\Gamma$ which contains $F_x$, 
and $S_i\subset U$ be simple polygonal paths which connect each $p_i\succ x$ to $x$ without intersecting each other and  $\Gamma$; see the middle diagram in Figure \ref{fig:leaf}. 
We are going to reindex $S_i$ and $p_i$ as follows. Let $\sigma$ be a circle centered at $x$  with sufficiently small radius so that it intersects $\Gamma$ only twice, and each $S_i$ only once. Orient $\sigma$ counterclockwise, reindex $S_i$, from $i=1,\dots, k$, in order that they intersect $\sigma\cap U$, and then reindex $p_i$ accordingly. 
\begin{figure}[h]
\centering
\begin{overpic}[height=1.7in]{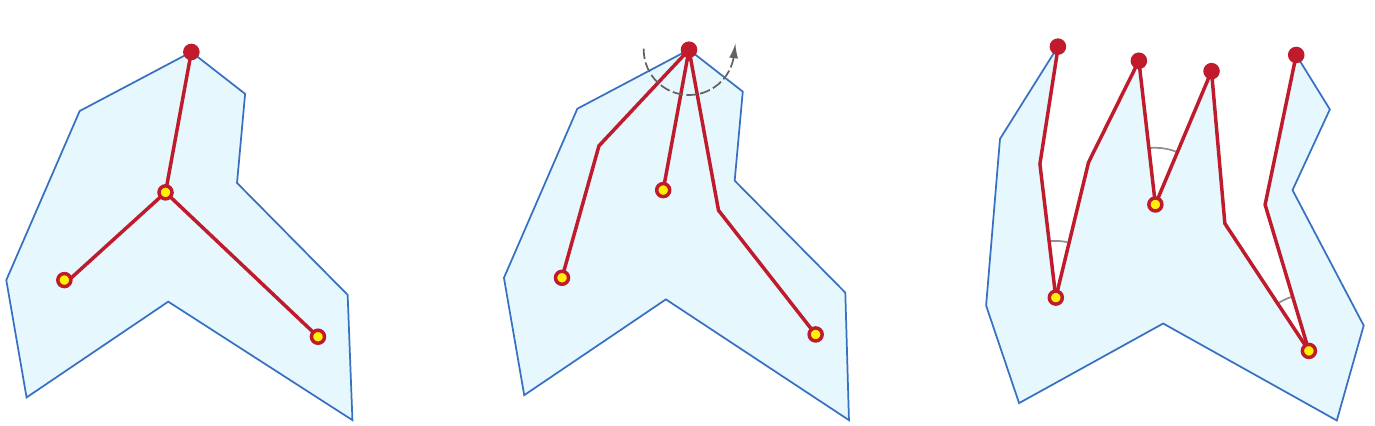}
\put(2.5,3.75){\Small$U$}
\put(1.5,17){\Small$\Gamma$}
\put(13.2,28.5){\Small$x$}
\put(39,14){\Small$S_1$}
\put(45.75,20){\Small$S_2$}
\put(52.5,10.5){\Small$S_3$}
\put(40,8.25){\Small$p_{1}$}
\put(47,14.75){\Small$p_{2}$}
\put(58.5,4.5){\Small$p_{3}$}
\put(74,3){\Small$U''$}
\put(76,29){\Small$x'_L$}
\put(93.5,28.25){\Small$x'_R$}
\put(38.5,3.75){\Small$U$}
\put(37.5,17){\Small$\Gamma$}
\put(70,17){\Small$\Gamma'$}
\put(49.5,28.5){\Small$x$}
\put(81,28.5){\Small$x''_{2,L}$}
\put(87,28){\Small$x''_{2,R}$}
\put(76,6.75){\Small$p''_{1}$}
\put(83.5,13.5){\Small$p''_{2}$}
\put(95,3.25){\Small$p''_{3}$}
\put(76.25,16){\tiny$\beta_1$}
\put(83.75,21){\tiny$\beta_2$}
\put(90.5,13){\tiny$\beta_3$}
\end{overpic}
\caption{}\label{fig:leaf}
\end{figure}

Now let $S:=\cup_{i}S_i$ be the resulting spanning tree for vertices of $U$. Then $\ol S:=f(S)$ is a  tree on $\ol U:=f(U)$. Let $\ol U_{\ol S}$ denote the topological disk obtained by cutting $\ol U$ along $\ol S$,  $\textup{cov}_S\colon \ol U_{\ol S}\to \ol U$  be the corresponding covering map, $u_S\colon \ol U_{\ol S}\to\R^2$ be an unfolding  given by Lemma \ref{lem:develop}, and define the multivalued mapping $\theta\colon U\to\R^2$ by
$$
\theta:=u_S\circ \textup{cov}_S^{-1}\circ f.
$$
Comparing this definition with that of $\psi$ given by \eqref{eq:phi} shows that  $\Gamma'':=\theta(\Gamma)$ is congruent to $\Gamma':=\psi(\Gamma)$. Indeed $\Gamma'$, $\Gamma''$ are determined, up to a translation, by the edge lengths of $\Gamma$ and its interior angles with respect to $U$. So we may assume that $\Gamma''=\Gamma'$ (by choosing $u_S$ appropriately, or composing it with an isometry of $\R^2$ which caries $\Gamma''$ to $\Gamma'$). Now note that $\theta$ converges to the identity map on $U$, just as $\psi$ does by Proposition \ref{prop:converge}. Thus, if we set $x'':=\theta(x)$, then $x''\to x$, as $\beta\to 0$.

Each  $p_i$ in $U$ has a single image $p_i''$ under $\theta$, while
there are two images of $x$ under $\theta$ corresponding to each $S_i$, which are denoted by $x_{i,L}''$ and $x_{i,R}''$; see the right diagram in Figure \ref{fig:leaf}. These may be defined similar to the way we defined $x_L'$ and $x_{R}'$, by extending $S$ to a triangulation of $U$.
We claim that
\begin{equation}\label{eq:tildec2}
\tilde{c}_x=\alpha_x^{-1}\sum_{i\succeq x}\alpha_i\tilde p_i,\quad\text{where}\quad\tilde p_i:=x+\frac{J(x''_{i,R}-x''_{i,L})}{\beta_i}.
\end{equation}
Indeed, since $\Gamma''=\Gamma'$, and due to our reindexing of $S_i$, we have $x'_L=x_{1,L}''$, $x'_R=x_{k,R}''$, and 
$x_{i,R}''=x_{i+1,L}''$, for $1\leq i\leq k-1$. Thus, since $\beta_i=\beta\alpha_i=\beta_x\alpha_x^{-1}\alpha_i$, 
$$
 x'_R-x'_L=\sum_{i\succeq x}(x_{i,R}''-x_{i,L}'')=\sum_{i\succeq x}\beta_iJ(x-\tilde p_i)=J(\beta_x x-\sum_{i\succeq x}\beta_i\tilde p_i)=\beta_x J(x-\tilde c_x).
$$ 
Applying $J$ to the far left and right sides of the last expression yields \eqref{eq:tildec2}. Next note that, since $\angle x''_{i,L}p_i''x''_{i,R}=\beta_i$, elementary trigonometry yields that
$$
p_i''=x''_{i,M}+\frac{J(x''_{i,R}-x''_{i,L})}{2\tan(\beta_i/2)}, \quad\text{where}\quad x''_{i,M}=\frac{x''_{i,L}+x''_{i,R}}{2};
$$
\begin{figure}[h]
\centering
\begin{overpic}[height=1in]{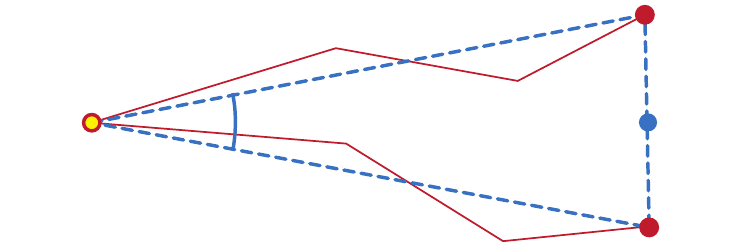}
\put(7,16.25){\Small$p_i''$}
\put(32.25,16.75){\Small$\beta_i$}
\put(88.25,30.75){\Small$x_{i,L}''$}
\put(88.5,2.25){\Small$x_{i,R}''$}
\put(88.5,16.25){\Small$x_{i,M}''$}
\put(53.5,2){\Small$S_{i,R}''$}
\put(54,20.25){\Small$S_{i,L}''$}
\end{overpic}
\caption{}\label{fig:angle}
\end{figure}
see Figure \ref{fig:angle}. Thus we have
$$
\tilde p_i-p_i=(x_{i,M}''-p_i'')(1 -2\tan(\beta_i/2)/\beta_i)+(x-x_{i,M}'')+(p_i''-p_i).
$$
Since the right hand side vanishes, as $\beta\to 0$, it follows that $\tilde p_i\to p_i$, and consequently $\tilde c_x\to c_x$ as desired. 
\end{proof}

\begin{note}\label{note:pivot}
Let $R_{p,\theta}\colon\R^2\to\R^2$ denote the clockwise rotation about the point $p$ by the angle $\theta$. As we discussed in the proof of Lemma \ref{lem:T1}, $x''_{i,R}=R_{p_i'',\beta_i}(x''_{i,L})$. So, since $p_i''\to p_i$, as $\beta\to 0$,
$$
x'_R=R_{p_1'',\beta_1}\circ\dots\circ R_{p_k'',\beta_k}(x'_L)\to R_{p_1,\beta_1}\circ\dots\circ R_{p_k,\beta_k}(x'_L)=R_{p_x,\theta_x}(x'_L),
$$
for some $p_x\in\R^2$ and  $\theta_x\in[0,2\pi)$. It is known  that 
\cite[Lem. 1]{orourke2017b}, as $\beta\to 0$,
$
\theta\to \sum_{i=1}^k\beta_i=\sum_{i\succeq x}\beta_i=\beta_x, 
$
and
$$
 p_x\to \beta_x^{-1}\sum_{i=1}^k\beta_ip_i=\alpha_x^{-1}\sum_{i\succeq x}\alpha_ip_i=c_x.
$$
So $c_x$ is the limit of the cumulative pivot point of descendant vertices of $x$, which is the justification for the term ``center of rotation". See also \cite{chen:flat} for a study of unfoldings of flat regions of convex polyhedra.
\end{note}

\begin{proof}[Proof of Theorem \ref{thm:F}]
Fix a cut forest  $F$  of $G$.  Then there exists an interior vertex $p_i$ of $G$ which does not satisfy \eqref{eq:cxe} and will be fixed henceforth. Let $\Phi_R$ and $\Phi_L$ be faces of $G$ which lie to the right and left of the oriented edge $p_ip_i^*$ respectively, see Figure \ref{fig:kite}. 
\begin{figure}[h]
\centering
\begin{overpic}[height=1.5in]{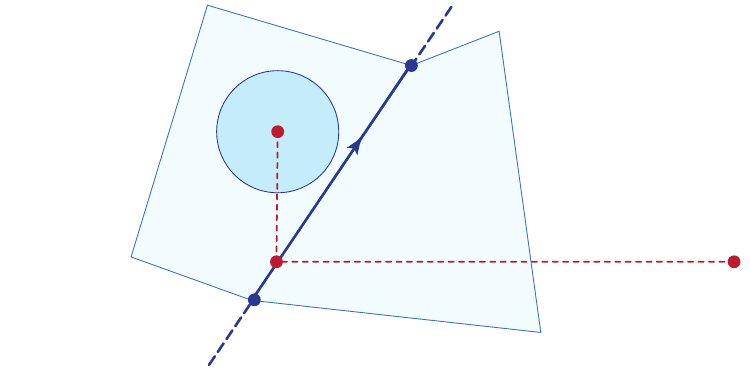}
\put(34,6.5){\small $p_i$}
\put(54,37){\small $p_i^*$}
\put(19,-1){\small $\ol{p_ip_i^*}$}
\put(36.5,11.4){\small $x$}
\put(38.25,32.5){\small $y$}
\put(31,28){\small $D$}
\put(97,11){\small $c_x$}
\put(20,16){\small $\Phi_L$}
\put(64,8){\small $\Phi_R$}
\end{overpic}
\caption{}\label{fig:kite}
\end{figure}
We will show that, for $\beta$  sufficiently small (i.e. $0<\beta\leq\beta_0(F)$), a point of $\Phi_R'$ lies in the interior of $\Phi_L'$. So $C'_{\beta,F}$ is not simple. Since $G$ admits only finitely many cut forests, this will complete the proof (once we let $\beta$ be smaller than the minimum value of $\beta_0(F)$ as $F$ ranges over all cut forests of $G$).
To start, fix $\lambda>0$ so small that
\begin{equation}\label{eq:pipj}
\left\l p_i^*-p_i,p_i-c_i\right\r\leq-\lambda |p_i^*-p_i|.
\end{equation}
Let $x$  be an interior point of $p_ip_i^*$ such that
\begin{equation}\label{eq:a-pi}
|x-p_i|\leq\lambda/2,
\end{equation}
and note that $x$ does not depend on $\beta$. Next  set
\begin{equation*}\label{eq:b}
y:=x+\beta_xJ(x-c_x), \quad\quad\quad r:=\beta_x\lambda/2,
\end{equation*}
and let $D$ be the disk of radius $r$ centered at $y$.
We will show that for $\beta$ small: $|y'-x'_R|< r$, and $D\subset \Phi_L$. Thus $x'_R\in \inte(D')\subset \inte(\Phi_L')$, as desired.

By the triangle inequality,
$$
|y'-x'_R|\leq|(x'_R-x'_L)-(y-x)| +|(y-x)-(y'-x'_L)|.
$$
 By Lemma \ref{lem:T1} we may choose $\beta$ so small that
 $|\tilde c_x-c_x|< \lambda/4$. Then, by \eqref{eq:tildec},
 $$
|(x'_R-x'_L)-(y-x)| = |\beta_xJ(x-\tilde c_x)-\beta_xJ(x-c_x)|=\beta_x|\tilde c_x-c_x|<\beta_x\lambda/4.
 $$
Set $\eta:=\diam(P)$. Note that, as $x$ lies in the interior of $p_ip_i^*$, we may choose $\beta$ so small that $y$ lies in the triangle $\Delta$ of $\Phi_L$ which rests on $p_ip_i^*$. Thus, by Proposition \ref{prop:converge},  we can  make sure that
  $$
 |(y-x)-(y'-x'_L)|\leq (\lambda/(4\eta)) |y-x|=(\lambda/(4\eta)) \beta_x|x-c_x|\leq\beta_x\lambda/4,
 $$
since $c_x\in P$ and therefore $|x-c_x|\leq\eta$.  The last three displayed expressions yield that
 $
 |y'-x'_R|<  \beta_x\lambda/2=r,
 $
as claimed.
 
As $\beta\to 0$, we have $y\to x$ and $r\to 0$. Thus, for $\beta$  small, $D\subset \Phi_R\cup \Phi_L$.
Since $y$ lies on the left side of the oriented line $\ol{p_ip_i^*}$ passing through $p_i$ and $p_i^*$,  it follows that $y\in \Phi_L$. So  it remains to check that $\dist(y, \ol{p_ip_i^*})\geq r$. 
By definition, $c_x=c_{p_i}=c_i$. Thus, by \eqref{eq:pipj} and \eqref{eq:a-pi},  
\begin{equation*}\label{eq:aca}
\l x-c_x,p_i^*-p_i\r=\l x-p_i, p_i^*-p_i\r+\l p_i-c_i, p_i^*-p_i\r\leq-(\lambda/2)|p_i^*-p_i|.
\end{equation*}
So
$$
\cos(\angle c_x x p_i^*)=-\frac{\l x-c_x,p_i^*-p_i\r}{|x-c_x||p_i^*-p_i|}\geq\frac{\lambda/2}{|x-c_x|}
=\frac{\beta_x\lambda/2}{\beta_x |x-c_x|}=\frac{r}{|y-x|},
$$
which yields
 $
  \dist(y, \ol{p_ip_i^*})=\sin(\angle yxp_i^*)|y-x|=\cos(\angle c_x x p_i^*)|y-x|\geq r,
 $
 and completes the proof.
\end{proof}

\section{A Non-monotone Convex Subdivision of \\the Equilateral Triangle}\label{sec:gadget}
A convex subdivision of a convex polygon  is \emph{weighted} if the corresponding graph $G$ is weighted, as defined in Section \ref{sec:pseudo}. Further the subdivision, or its graph $G$, is \emph{non-monotone} provided that $G$ admits no monotone cut forests, as defined in Section \ref{sec:cut}. In this section we show:
\begin{thm}\label{thm:gadget}
The equilateral triangle admits a non-monotone weighted convex subdivision with $84$ interior vertices.
\end{thm}

\noindent Earlier Tarasov \cite{tarasov} had constructed a non-monotone subdivision of the  equilateral triangle with over $500$ vertices.
Here we simplify that construction as follows.

\subsection{Coordinates and weights}
The edge graph $G$ of our subdivision  is illustrated in Figures \ref{fig:triangle}, 
  \begin{figure}[h]
\centering
\begin{overpic}[width=2.5in]{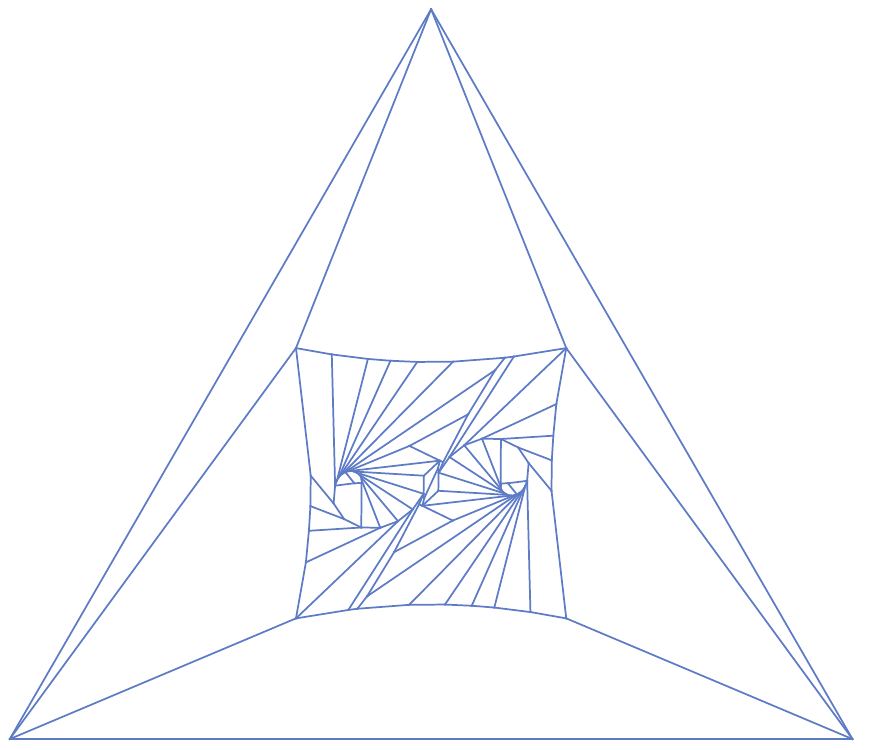}
\put(48,86){\small $a$}
\put(-2.1,-1){\small $b$}
\put(98,-1){\small $c$}
\end{overpic}
\caption{}\label{fig:triangle}
\end{figure}
and a larger depiction of the subgraph of $G$ spanned by its vertices in the interior of the triangle appears in Figure \ref{fig:square}. We call this subgraph \emph{the square}.
We assume that the triangle has vertices $a:=(0, 140 \sqrt{3} - 85)$, $b:=(-140, -85)$, and $c:=(140, -85)$.
The square is symmetric with respect to reflection through $o:=(0,0)$ and has $84$ vertices. We label half of these vertices  by $p_i$, $i=1,\dots, 42$, 
as shown in Figure \ref{fig:square}. 
The coordinates of these points are listed in Table \ref{tab:coordinates}, 
as well as in an accompanying \emph{Mathematica} notebook \cite{barvinok&ghomi:mathematica} that we have provided.
 The other vertices of the square are the reflection of $p_i$, and will be denoted by $p_{-i}:=-p_i$.   
We assume that $p_{\pm1}$ have equal weights, and the weights of all other vertices are  arbitrarily small. 
\begin{figure}[h]
\centering
\begin{overpic}[width=4.9in]{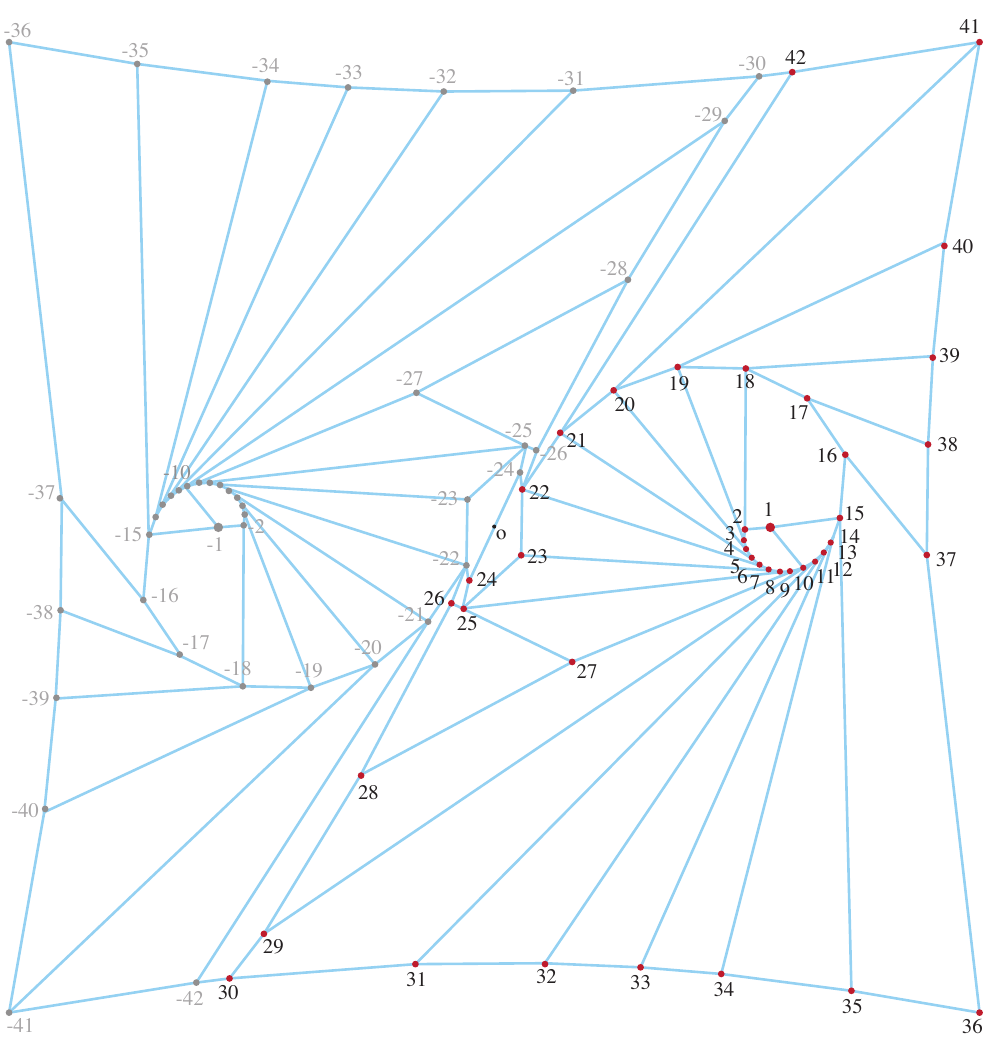}
\end{overpic}
\caption{}\label{fig:square}
\end{figure}

 \subsection{Main properties}\label{subsec:properties}
As in \cite{orourke2016}, we say a path $\Gamma=(v_1, \dots, v_n)$ of $G$ is \emph{radially monotone} with respect to a point $x$ provided that 
\begin{equation}\label{eq:rm}
 \angle(v_{i+1}, v_i, x)\geq \pi/2
\end{equation}
 for $i=1,\dots, n-1$. We say a path in $G$ is \emph{maximal} if it ends on the boundary of $G$. Our subdivision has been designed so that it has  two important features, as expressed in the following lemmas and illustrated in Figure \ref{fig:spirals}.

 \begin{lem}\label{lem:subpath1}
 Let  $\Gamma$ be a path in $G$ which originates at $p_1$ and is radially monotone with respect to $p_1$. Then $\Gamma$ must be a subpath of
 \begin{equation}\label{eq:Gamma1}
(p_1, p_{\ell},p_{\ell+1},\dots, p_{21},p_{22},p_{-24},p_{24}, p_{25}, p_{26}, p_{28}, p_{29}, p_{30}, p_{-42},p_{-41}, q),
\end{equation}
 where $\ell=2$, $10$, or $15$, and $q$ is a vertex of the triangle.
  \end{lem}
  \begin{proof}
  The only vertices of $G$ which are adjacent to $p_1$ are $p_\ell$, where $\ell=2$, $10$, or $15$; see the left diagram in Figure \ref{fig:paths}.
 Thus the second vertex of $\Gamma$ must be $p_\ell$. Suppose that $\ell=2$. Other than $p_1$, the only vertices adjacent to $p_2$  are $p_{18}$ and $p_3$. One quickly checks that $\angle(p_3,p_2, p_1)\geq \pi/2$ while $\angle(p_{18},p_2, p_1)<\pi/2$. Thus $p_3$ (or $p_{\ell+1}$) is the only choice for the next vertex of $\Gamma$. Similarly, to complete the proof, it suffices to check  that, if we denote the vertices of the proposed path \eqref{eq:Gamma1} by $v_i$, $i=1,\dots, n$, then $v_{i+1}$ is the only vertex $w$ adjacent to $v_i$ such that  $\angle(w,v_i, p_1)\geq \pi/2$. Finally, once the path reaches a vertex of the triangle, then it cannot be extended further, since the angles of the triangle are all obtuse. We omit the computations  since they are trivial, and refer the reader instead to the accompanying \emph{Mathematica} notebook \cite{barvinok&ghomi:mathematica}, where all the computations have been recorded. 
   \end{proof}
   
    \begin{table}[h]
{\tiny
\begin{align*}
 \mathbf{1} &: (25.5,0)     &  \mathbf{2} &: (23.1,-0.1) &  \mathbf{3} &: (23.1,-1.1) &  \mathbf{4} &: (23.3,-2.1) &\mathbf{5} &: (23.9,-2.9)\\
 \mathbf{6} &: (24.7,-3.5) &  \mathbf{7} &: (25.6,-3.9) &  \mathbf{8} &: (26.6,-4.1) &  \mathbf{9} &: (27.6,-4)     & \mathbf{10}&: (28.5,-3.7)\\
 \mathbf{11} &: (29.4,-3.2) & \mathbf{12} &: (30.2,-2.6)&\mathbf{13} &: (30.8,-1.8)    & \mathbf{14} &: (31.3,-0.9) & \mathbf{15} &: (31.9,0.8)\\
 \mathbf{16} &: (32.4,6.7) & \mathbf{17} &: (28.9,11.9) & \mathbf{18} &: (23.2,14.7) & \mathbf{19} &: (16.9,14.8) & \mathbf{20} &: (11,12.6)\\
\mathbf{21} &: (6.1,8.7)    & \mathbf{22} &: (2.5,3.5)     & \mathbf{23} &: (2.4,-2.5) & \mathbf{24} &: (-2.3,-5)         & \mathbf{25} &: (-2.9,-7.5)\\
 \mathbf{26} &: (-4,-7)       & \mathbf{27} &: (7.2,-12.4) & \mathbf{28} &: (-12.3,-23) & \mathbf{29} &: (-21.3,-37.6)    & \mathbf{30} &: (-24.5,-41.7) \\
\mathbf{31} &: (-7.3,-40.4) & \mathbf{32} &: (4.6,-40.3) & \mathbf{33} &: (13.5,-40.7)    & \mathbf{34} &: (21,-41.3) & \mathbf{35} &: (33,-42.8)  \\
\mathbf{36} &: (44.9,-44.9) & \mathbf{37} &: (40,-2.5)    & \mathbf{38} &: (40,7.6)     & \mathbf{39} &: (40.6,15.8)         & \mathbf{40} &:(41.6,26.3)\\
\mathbf{41} &: (44.9,44.9)    & \mathbf{42} &: (27.5,42.1)
\end{align*}
}\caption{}\label{tab:coordinates}
\end{table}

     \begin{lem}\label{lem:subpath2}
Let $\Gamma$ be a maximal path in $G$ which originates at $p_{24}$ and is radially monotone with respect to $o$.  Then  
$$
 \Gamma=(p_{24},p_{25},p_{27},p_{9},p_{10},\dots, p_{15},\dots).
$$ 
In particular $\Gamma$ contains $p_{15}$.
 \end{lem}
 \begin{proof}
As in the proof of Lemma \ref{lem:subpath1}, one may easily check that the only possible choice for the successor of $p_{24}$ in $\Gamma$, which would satisfy \eqref{eq:rm}, is $p_{25}$. Similarly, one may recover all other vertices as well; see the right diagram in Figure \ref{fig:paths}.
 This involves a series of trivial computations which are included in the accompanying \emph{Mathematica} notebook \cite{barvinok&ghomi:mathematica}.
 \end{proof}

    \begin{figure}[h]
\centering
\begin{overpic}[height=1.5in]{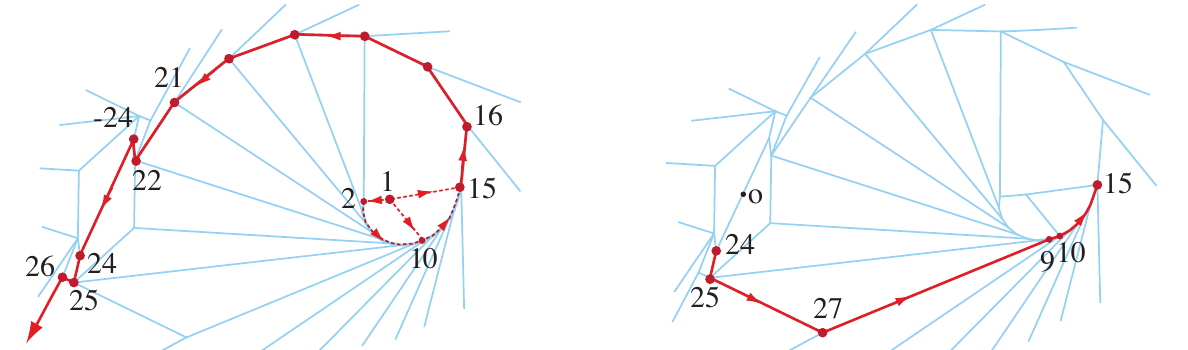}
\end{overpic}
\caption{}\label{fig:paths}
\end{figure}

 \subsection{Proof of Theorem \ref{thm:gadget}}
 
 We claim that the subdivision generated by the weighted graph $G$ described above is non-monotone. Assume, towards a contradiction, that  $G$  admits a monotone cut forest $F$ (as defined in Section \ref{subsec:cutforest}). Then each vertex $p_i$ of the square part of $G$ has a (unique) ancestral path in $F$, which we denote by $\Gamma_{i}$. Recall that the final vertex of $\Gamma_i$ must be a vertex of the triangle, while all other vertices, which we call the interior vertices of $\Gamma_i$, lie in the interior of the triangle, or the square part. 
   
First note that $\Gamma_1$ and $\Gamma_{-1}$ must share an interior vertex.  If not, then $p_{-1}$ cannot be a descendant of any interior vertex  of $\Gamma_1$. Consequently, the center of rotation of  each interior vertex of $\Gamma_1$  remains arbitrarily close to $p_1$, since all vertices of $G$ other than $p_{\pm 1}$ have arbitrarily small weights by assumption. It follows then (via condition \eqref{eq:cxe2}) that $\Gamma_1$ is almost radially monotone with respect to $p_1$, i.e., vertices of $\Gamma_1$ satisfy condition \eqref{eq:rm}  for $x=p_1$ and $\pi/2-\epsilon$, for arbitrarily small $\epsilon>0$. Thus  $\Gamma_1$ must be radially monotone with respect to $p_1$, since there are only finitely many paths in $G$, and by convexity there exist radially monotone paths, with respect to any given point, which emanate from any given vertex of $G$. 
So, by Lemma \ref{lem:subpath1}, $\Gamma_1$ contains $p_{\pm 24}$. By symmetry, $\Gamma_{-1}$ must contain these vertices as well, since the interior vertices of $\Gamma_{-1}$ are  the reflections of interior vertices of $\Gamma_1$.
 So $\Gamma_{1}$ and $\Gamma_{-1}$  must  join at some interior vertex.

  Let $p_{m}$ be the first (interior) vertex where $\Gamma_1$ and $\Gamma_{-1}$ join. Then, as we described above, the subpath of $\Gamma_{1}$ from $p_{1}$ to $p_m$ will be radially monotone with respect to $p_1$. So, by Lemma \ref{lem:subpath1}, it must be a (proper) subpath of \eqref{eq:Gamma1}. Similarly, by symmetry, the subpath of $\Gamma_{-1}$ from $p_{-1}$ to $p_m$ must be the reflection of a  subpath of \eqref{eq:Gamma1}.
Hence the only possibilities are: $m=24$ or $m=-24$. After a reflection, we may assume that $m=24$; see  Figure \ref{fig:spirals}. Then $p_1$ and $p_{-1}$ are both descendants of $p_{24}$; therefore, the center of rotation of any  vertex of $\Gamma_{24}$ is  arbitrarily close to the center of mass of $p_1$ and $p_{-1}$, which is $o$. So $\Gamma_{24}$  is radially monotone with respect to $o$. Consequently, by Lemma \ref{lem:subpath2}, $\Gamma_{24}$ contains  $p_{15}$. But $\Gamma_{24}$ is a subpath of $\Gamma_{1}$, which already passes through $p_{15}$ prior to reaching $p_{24}$.  Thus $\Gamma_{1}$ contains a  cycle, which is the desired contradiction, and completes the proof of Theorem \ref{thm:gadget}.

  \begin{figure}[h]
\centering
\begin{overpic}[height=1.6in]{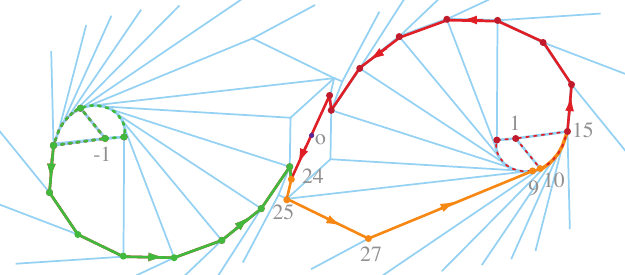}
\put(85,40){$\Gamma_1$}
\put(9,2){$\Gamma_{-1}$}
\end{overpic}
\caption{}\label{fig:spirals}
\end{figure}

Theorem \ref{thm:gadget} together with Theorem \ref{thm:F} now immediately yields:

\begin{cor}\label{cor:gadget}
There exists a convex cap $C$ over the equilateral triangle with $84$ interior vertices and a pseudo-edge graph with respect to which $C$ is not unfoldable. Furthermore, the total curvature of $C$ may be arbitrarily small.\qed
\end{cor}

\section{Proof of Theorem \ref{thm:main}}\label{sec:proof}

We need only one more observation. A \emph{simple arc} $\Gamma$ in a topological space $X$ is the image of a continuous  mapping $\gamma\colon[a,b]\to X$, which is one-to-one on $(a,b)$. We say $\Gamma$ is a \emph{loop} provided that $\gamma(a)=\gamma(b)$. The following basic fact is also used in \cite{bern-demain-epstein} to construct examples of un-unfoldable polyhedra. We omit the proof since it is fairly trivial (e.g. it follows by considering the different combinations). 

\begin{lem}\label{lem:tetrahedron}
Let $E\subset\S^2$ be an embedded graph which is isomorphic to the edge graph of a tetrahedron. Suppose there exists a simple arc $\Gamma_i$ in each face $\Phi_i$ of $E$ whose end points are distinct 
vertices  of $\Phi_i$, and whose interior lies in the interior of $\Phi_i$. Then $\Gamma:=\cup_i \Gamma_i$ contains a loop.
\end{lem}

Now we are ready to prove the main result of this work:

\begin{proof}[Proof of Theorem \ref{thm:main}]
Let $C$ be the convex cap over the equilateral triangle given by Corollary \ref{cor:gadget}. By Lemma \ref{lem:BetaToZero} we may assume that the curvature $\kappa(C)$ is so small that the total angles of $C$ at each of its boundary vertices is less than $2\pi/3$. Let $C_i$, $i=1,\dots, 4$, be congruent copies of $C$ positioned over the faces of a regular tetrahedron, $K:=\cup_i C_i$, and $E$ be the union of the pseudo-edges $E_i$ of $C_i$. Since $C_i$ have $84$ interior vertices each,  $K$ has $340$ vertices. We claim that $K$ is not unfoldable with respect to $E$. To see this  suppose that $T$ is a spanning tree of $E$, and let $F_i$ be the closure of the restriction of $T$ to the interior of $C_i$. If each $F_i$ contains an arc connecting a pair of boundary vertices of $C_i$, then, by Lemma \ref{lem:tetrahedron}, $T$ must contain a loop (or a cycle) which is not possible. Thus  $F_j$ must form a spanning forest of $E_j$ for some $1\leq j\leq 4$.
Consequently, by Corollary \ref{cor:gadget}, the unfolding of $C_j$ with respect to $F_j$ is not simple. Hence the unfolding of $K$ with respect to $T$ is not simple, which completes the proof.
\end{proof}

\begin{note}
The obvious question at the conclusion of this work is whether the above construction may yield  a counterexample to (the original form of) D\"{u}rer's conjecture. The answer would depend on whether the partition of the equilateral triangle in Section \ref{sec:gadget}, or some variation of it, can be \emph{lifted} to a convex cap, i.e., whether there exists a convex cap over the equilateral triangle whose edges (not only vertices) project onto the edges of the partition. If so, these caps would generate an edge un-unfoldable polyhedron when assembled on the faces of a tetrahedron, as described above. It is well-known that not every convex partition of a convex polygon can be lifted \cite[p. 56]{deloera-santos}. A necessary and sufficient condition, called the \emph{Maxwell-Cremona correspondence} \cite{schulz,whiteley, izmestiev2017statics}, is that the edges of the subdivision admit an equilibrium stress, which can be expressed as a system of linear equations. Thus, to produce a counterexample to D\"{u}rer's conjecture (if one exists) it would suffice to find a non-monotone subdivision of the equilateral triangle which satisfies these equations. Alternatively, if  no such subdivision exists, then that would yield more evidence in support of the conjecture.
\end{note}

\section*{Acknowledgements}
Our debt to the original investigations of  A. Tarasov \cite{tarasov} is evident throughout this work. Thanks also to J. O'Rourke for his interest and useful comments on earlier drafts of this paper. Furthermore we are grateful to several anonymous reviewers who prompted us to clarify the exposition of this work. Parts of this work were completed while the first named author participated in the REU program in the School of Math at Georgia Tech in the Summer of 2017.

\bibliographystyle{abbrv}

\begin{bibdiv}
\begin{biblist}

\bib{alexandrov:polyhedra}{book}{
      author={Alexandrov, A.~D.},
       title={Convex polyhedra},
      series={Springer Monographs in Mathematics},
   publisher={Springer-Verlag},
     address={Berlin},
        date={2005},
        ISBN={3-540-23158-7},
        note={Translated from the 1950 Russian edition by N. S. Dairbekov, S.
  S. Kutateladze and A. B. Sossinsky, With comments and bibliography by V. A.
  Zalgaller and appendices by L. A. Shor and Yu. A. Volkov},
      review={\MR{MR2127379 (2005j:52002)}},
}

\bib{barvinok&ghomi:mathematica}{article}{
      author={Barvinok, Nicholas},
      author={Ghomi, Mohammad},
       title={Pseudo-edge.nb},
        date={2017},
     journal={Mathematica Package, available at
  \url{http://www.math.gatech.edu/~ghomi/MathematicaNBs/Pseudo-Edge.nb}},
}

\bib{bern-demain-epstein}{article}{
      author={Bern, Marshall},
      author={Demaine, Erik~D.},
      author={Eppstein, David},
      author={Kuo, Eric},
      author={Mantler, Andrea},
      author={Snoeyink, Jack},
       title={Ununfoldable polyhedra with convex faces},
        date={2003},
        ISSN={0925-7721},
     journal={Comput. Geom.},
      volume={24},
      number={2},
       pages={51\ndash 62},
         url={https://doi.org/10.1016/S0925-7721(02)00091-3},
        note={Special issue on the Fourth CGC Workshop on Computational
  Geometry (Baltimore, MD, 1999)},
      review={\MR{1956035}},
}

\bib{bobenko-izmestiev}{article}{
      author={Bobenko, Alexander~I.},
      author={Izmestiev, Ivan},
       title={Alexandrov's theorem, weighted {D}elaunay triangulations, and
  mixed volumes},
        date={2008},
        ISSN={0373-0956},
     journal={Ann. Inst. Fourier (Grenoble)},
      volume={58},
      number={2},
       pages={447\ndash 505},
         url={http://aif.cedram.org/item?id=AIF_2008__58_2_447_0},
      review={\MR{2410380}},
}

\bib{chen:flat}{thesis}{
      author={Chen, Yanping},
       title={Edge-unfolding almost-flat convex polyhedral terrains},
        type={Master's Thesis},
        date={2013},
}

\bib{cfg}{book}{
      author={Croft, Hallard~T.},
      author={Falconer, Kenneth~J.},
      author={Guy, Richard~K.},
       title={Unsolved problems in geometry},
      series={Problem Books in Mathematics},
   publisher={Springer-Verlag},
     address={New York},
        date={1991},
        ISBN={0-387-97506-3},
         url={http://dx.doi.org/10.1007/978-1-4612-0963-8},
        note={Unsolved Problems in Intuitive Mathematics, II},
      review={\MR{1107516 (92c:52001)}},
}

\bib{deloera-santos}{book}{
      author={De~Loera, Jes\'{u}s~A.},
      author={Rambau, J\"{o}rg},
      author={Santos, Francisco},
       title={Triangulations},
      series={Algorithms and Computation in Mathematics},
   publisher={Springer-Verlag, Berlin},
        date={2010},
      volume={25},
        ISBN={978-3-642-12970-4},
  url={https://doi-org.prx.library.gatech.edu/10.1007/978-3-642-12971-1},
        note={Structures for algorithms and applications},
      review={\MR{2743368}},
}

\bib{do:book}{book}{
      author={Demaine, Erik~D.},
      author={O'Rourke, Joseph},
       title={Geometric folding algorithms},
   publisher={Cambridge University Press},
     address={Cambridge},
        date={2007},
        ISBN={978-0-521-85757-4; 0-521-85757-0},
         url={http://dx.doi.org/10.1017/CBO9780511735172},
        note={Linkages, origami, polyhedra},
      review={\MR{2354878 (2008g:52001)}},
}

\bib{durer}{book}{
      author={D\"{u}rer, Albrecht},
       title={The painter's manual: A manual of measurement of lines, areas,
  and solids by means of compass and ruler assembled by albrecht d\"{u}rer for
  the use of all lovers of art with appropriate illustrations arranged to be
  printed in the year mdxxv},
   publisher={Abaris Books, New York, N.Y.},
        date={1977 (1525)},
}

\bib{ghomi:verticesA}{article}{
      author={Ghomi, Mohammad},
       title={A {R}iemannian four vertex theorem for surfaces with boundary},
        date={2011},
        ISSN={0002-9939},
     journal={Proc. Amer. Math. Soc.},
      volume={139},
      number={1},
       pages={293\ndash 303},
         url={http://dx.doi.org/10.1090/S0002-9939-2010-10507-5},
      review={\MR{2729091 (2011k:53038)}},
}

\bib{ghomi:durer}{article}{
      author={Ghomi, Mohammad},
       title={Affine unfoldings of convex polyhedra},
        date={2014},
        ISSN={1465-3060},
     journal={Geom. Topol.},
      volume={18},
      number={5},
       pages={3055\ndash 3090},
         url={http://dx.doi.org/10.2140/gt.2014.18.3055},
      review={\MR{3285229}},
}

\bib{ghomi:notices}{article}{
      author={Ghomi, Mohammad},
       title={D\"{u}rer's unfolding problem for convex polyhedra},
        date={2018},
     journal={Notices Amer. Math. Soc.},
      volume={65},
       pages={25\ndash 27},
}

\bib{grunbaum}{article}{
      author={Gr{\"u}nbaum, Branko},
       title={Nets of polyhedra. {II}},
        date={1991},
        ISSN={1065-7371},
     journal={Geombinatorics},
      volume={1},
      number={3},
       pages={5\ndash 10},
      review={\MR{1208436}},
}

\bib{izmestiev2017statics}{article}{
      author={Izmestiev, Ivan},
       title={Statics and kinematics of frameworks in euclidean and
  non-euclidean geometry},
        date={2017},
     journal={arXiv preprint arXiv:1707.02172},
}

\bib{lubiw-orourke2017}{article}{
      author={Lubiw, Anna},
      author={O'Rourke, Joseph},
       title={Angle-monotone paths in non-obtuse triangulations},
        date={2017},
     journal={arXiv preprint arXiv:1707.00219},
}

\bib{orourke:book}{book}{
      author={O'Rourke, Joseph},
       title={How to fold it},
   publisher={Cambridge University Press, Cambridge},
        date={2011},
        ISBN={978-0-521-14547-3},
         url={http://dx.doi.org/10.1017/CBO9780511975028},
        note={The mathematics of linkages, origami, and polyhedra},
      review={\MR{2768138}},
}

\bib{orourke2016}{article}{
      author={O'Rourke, Joseph},
       title={Unfolding convex polyhedra via radially monotone cut trees},
        date={2016},
     journal={arXiv preprint arXiv:1607.07421},
}

\bib{orourke2017b}{article}{
      author={O'Rourke, Joseph},
       title={Addendum to: Edge-unfolding nearly flat convex caps},
        date={2017},
     journal={arXiv preprint arXiv:1709.02433},
}

\bib{orourke2017}{article}{
      author={O'Rourke, Joseph},
       title={Edge-unfolding nearly flat convex caps},
        date={2017},
     journal={arXiv preprint arXiv:1707.01006},
}

\bib{pak:book}{article}{
      author={Pak, Igor},
       title={Lectures on discrete and polyhedral geometry},
        date={2008},
}

\bib{pogorelov:book}{book}{
      author={Pogorelov, A.~V.},
       title={Extrinsic geometry of convex surfaces},
   publisher={American Mathematical Society},
     address={Providence, R.I.},
        date={1973},
        note={Translated from the Russian by Israel Program for Scientific
  Translations, Translations of Mathematical Monographs, Vol. 35},
      review={\MR{MR0346714 (49 \#11439)}},
}

\bib{schulz}{thesis}{
      author={schulz, Andr\'{e}},
       title={Lifting planar graphs to realize integral 3-polytopes and topics
  in pseudo-triangulations},
        type={Ph.D. Thesis},
        date={2008},
}

\bib{shephard}{article}{
      author={Shephard, G.~C.},
       title={Convex polytopes with convex nets},
        date={1975},
        ISSN={0305-0041},
     journal={Math. Proc. Cambridge Philos. Soc.},
      volume={78},
      number={3},
       pages={389\ndash 403},
      review={\MR{0390915 (52 \#11738)}},
}

\bib{tarasov}{article}{
      author={Tarasov, Alexey},
       title={Existence of a polyhedron which does not have a non-overlapping
  pseudo-edge unfolding},
        date={2008},
     journal={arXiv:0806.2360v3},
}

\bib{whiteley}{article}{
      author={Whiteley, Walter},
       title={Motions and stresses of projected polyhedra},
        date={1982},
        ISSN={0226-9171},
     journal={Structural Topology},
      number={7},
       pages={13\ndash 38},
        note={With a French translation},
      review={\MR{721947}},
}

\bib{ziegler:book}{book}{
      author={Ziegler, G{\"u}nter~M.},
       title={Lectures on polytopes},
      series={Graduate Texts in Mathematics},
   publisher={Springer-Verlag, New York},
        date={1995},
      volume={152},
        ISBN={0-387-94365-X},
         url={http://dx.doi.org/10.1007/978-1-4613-8431-1},
      review={\MR{1311028}},
}

\end{biblist}
\end{bibdiv}


\begin{thebibliography}{10}
\bib{alexandrov:polyhedra}{book}{
      author={Alexandrov, A.~D.},
       title={Convex polyhedra},
      series={Springer Monographs in Mathematics},
   publisher={Springer-Verlag},
     address={Berlin},
        date={2005},
        ISBN={3-540-23158-7},
        note={Translated from the 1950 Russian edition by N. S. Dairbekov, S.
  S. Kutateladze and A. B. Sossinsky, With comments and bibliography by V. A.
  Zalgaller and appendices by L. A. Shor and Yu. A. Volkov},
      review={\MR{MR2127379 (2005j:52002)}},
}

\bib{barvinok&ghomi:mathematica}{article}{
      author={Barvinok, Nicholas},
      author={Ghomi, Mohammad},
       title={Pseudo-edge.nb},
        date={2017},
     journal={Mathematica Package, available at
  \url{http://www.math.gatech.edu/~ghomi/MathematicaNBs/Pseudo-Edge.nb}},
}

\bib{bern-demain-epstein}{article}{
      author={Bern, Marshall},
      author={Demaine, Erik~D.},
      author={Eppstein, David},
      author={Kuo, Eric},
      author={Mantler, Andrea},
      author={Snoeyink, Jack},
       title={Ununfoldable polyhedra with convex faces},
        date={2003},
        ISSN={0925-7721},
     journal={Comput. Geom.},
      volume={24},
      number={2},
       pages={51\ndash 62},
         url={https://doi.org/10.1016/S0925-7721(02)00091-3},
        note={Special issue on the Fourth CGC Workshop on Computational
  Geometry (Baltimore, MD, 1999)},
      review={\MR{1956035}},
}

\bib{bobenko-izmestiev}{article}{
      author={Bobenko, Alexander~I.},
      author={Izmestiev, Ivan},
       title={Alexandrov's theorem, weighted {D}elaunay triangulations, and
  mixed volumes},
        date={2008},
        ISSN={0373-0956},
     journal={Ann. Inst. Fourier (Grenoble)},
      volume={58},
      number={2},
       pages={447\ndash 505},
         url={http://aif.cedram.org/item?id=AIF_2008__58_2_447_0},
      review={\MR{2410380}},
}

\bib{chen:flat}{thesis}{
      author={Chen, Yanping},
       title={Edge-unfolding almost-flat convex polyhedral terrains},
        type={Master's Thesis},
        date={2013},
}

\bib{cfg}{book}{
      author={Croft, Hallard~T.},
      author={Falconer, Kenneth~J.},
      author={Guy, Richard~K.},
       title={Unsolved problems in geometry},
      series={Problem Books in Mathematics},
   publisher={Springer-Verlag},
     address={New York},
        date={1991},
        ISBN={0-387-97506-3},
         url={http://dx.doi.org/10.1007/978-1-4612-0963-8},
        note={Unsolved Problems in Intuitive Mathematics, II},
      review={\MR{1107516 (92c:52001)}},
}

\bib{deloera-santos}{book}{
      author={De~Loera, Jes\'{u}s~A.},
      author={Rambau, J\"{o}rg},
      author={Santos, Francisco},
       title={Triangulations},
      series={Algorithms and Computation in Mathematics},
   publisher={Springer-Verlag, Berlin},
        date={2010},
      volume={25},
        ISBN={978-3-642-12970-4},
  url={https://doi-org.prx.library.gatech.edu/10.1007/978-3-642-12971-1},
        note={Structures for algorithms and applications},
      review={\MR{2743368}},
}

\bib{do:book}{book}{
      author={Demaine, Erik~D.},
      author={O'Rourke, Joseph},
       title={Geometric folding algorithms},
   publisher={Cambridge University Press},
     address={Cambridge},
        date={2007},
        ISBN={978-0-521-85757-4; 0-521-85757-0},
         url={http://dx.doi.org/10.1017/CBO9780511735172},
        note={Linkages, origami, polyhedra},
      review={\MR{2354878 (2008g:52001)}},
}

\bib{durer}{book}{
      author={D\"{u}rer, Albrecht},
       title={The painter's manual: A manual of measurement of lines, areas,
  and solids by means of compass and ruler assembled by albrecht d\"{u}rer for
  the use of all lovers of art with appropriate illustrations arranged to be
  printed in the year mdxxv},
   publisher={Abaris Books, New York, N.Y.},
        date={1977 (1525)},
}

\bib{ghomi:verticesA}{article}{
      author={Ghomi, Mohammad},
       title={A {R}iemannian four vertex theorem for surfaces with boundary},
        date={2011},
        ISSN={0002-9939},
     journal={Proc. Amer. Math. Soc.},
      volume={139},
      number={1},
       pages={293\ndash 303},
         url={http://dx.doi.org/10.1090/S0002-9939-2010-10507-5},
      review={\MR{2729091 (2011k:53038)}},
}

\bib{ghomi:durer}{article}{
      author={Ghomi, Mohammad},
       title={Affine unfoldings of convex polyhedra},
        date={2014},
        ISSN={1465-3060},
     journal={Geom. Topol.},
      volume={18},
      number={5},
       pages={3055\ndash 3090},
         url={http://dx.doi.org/10.2140/gt.2014.18.3055},
      review={\MR{3285229}},
}

\bib{ghomi:notices}{article}{
      author={Ghomi, Mohammad},
       title={D\"{u}rer's unfolding problem for convex polyhedra},
        date={2018},
     journal={Notices Amer. Math. Soc.},
      volume={65},
       pages={25\ndash 27},
}

\bib{grunbaum}{article}{
      author={Gr{\"u}nbaum, Branko},
       title={Nets of polyhedra. {II}},
        date={1991},
        ISSN={1065-7371},
     journal={Geombinatorics},
      volume={1},
      number={3},
       pages={5\ndash 10},
      review={\MR{1208436}},
}

\bib{izmestiev2017statics}{article}{
      author={Izmestiev, Ivan},
       title={Statics and kinematics of frameworks in euclidean and
  non-euclidean geometry},
        date={2017},
     journal={arXiv preprint arXiv:1707.02172},
}

\bib{lubiw-orourke2017}{article}{
      author={Lubiw, Anna},
      author={O'Rourke, Joseph},
       title={Angle-monotone paths in non-obtuse triangulations},
        date={2017},
     journal={arXiv preprint arXiv:1707.00219},
}

\bib{orourke:book}{book}{
      author={O'Rourke, Joseph},
       title={How to fold it},
   publisher={Cambridge University Press, Cambridge},
        date={2011},
        ISBN={978-0-521-14547-3},
         url={http://dx.doi.org/10.1017/CBO9780511975028},
        note={The mathematics of linkages, origami, and polyhedra},
      review={\MR{2768138}},
}

\bib{orourke2016}{article}{
      author={O'Rourke, Joseph},
       title={Unfolding convex polyhedra via radially monotone cut trees},
        date={2016},
     journal={arXiv preprint arXiv:1607.07421},
}

\bib{orourke2017b}{article}{
      author={O'Rourke, Joseph},
       title={Addendum to: Edge-unfolding nearly flat convex caps},
        date={2017},
     journal={arXiv preprint arXiv:1709.02433},
}

\bib{orourke2017}{article}{
      author={O'Rourke, Joseph},
       title={Edge-unfolding nearly flat convex caps},
        date={2017},
     journal={arXiv preprint arXiv:1707.01006},
}

\bib{pak:book}{article}{
      author={Pak, Igor},
       title={Lectures on discrete and polyhedral geometry},
        date={2008},
}

\bib{pogorelov:book}{book}{
      author={Pogorelov, A.~V.},
       title={Extrinsic geometry of convex surfaces},
   publisher={American Mathematical Society},
     address={Providence, R.I.},
        date={1973},
        note={Translated from the Russian by Israel Program for Scientific
  Translations, Translations of Mathematical Monographs, Vol. 35},
      review={\MR{MR0346714 (49 \#11439)}},
}

\bib{schulz}{thesis}{
      author={schulz, Andr\'{e}},
       title={Lifting planar graphs to realize integral 3-polytopes and topics
  in pseudo-triangulations},
        type={Ph.D. Thesis},
        date={2008},
}

\bib{shephard}{article}{
      author={Shephard, G.~C.},
       title={Convex polytopes with convex nets},
        date={1975},
        ISSN={0305-0041},
     journal={Math. Proc. Cambridge Philos. Soc.},
      volume={78},
      number={3},
       pages={389\ndash 403},
      review={\MR{0390915 (52 \#11738)}},
}

\bib{tarasov}{article}{
      author={Tarasov, Alexey},
       title={Existence of a polyhedron which does not have a non-overlapping
  pseudo-edge unfolding},
        date={2008},
     journal={arXiv:0806.2360v3},
}

\bib{whiteley}{article}{
      author={Whiteley, Walter},
       title={Motions and stresses of projected polyhedra},
        date={1982},
        ISSN={0226-9171},
     journal={Structural Topology},
      number={7},
       pages={13\ndash 38},
        note={With a French translation},
      review={\MR{721947}},
}

\bib{ziegler:book}{book}{
      author={Ziegler, G{\"u}nter~M.},
       title={Lectures on polytopes},
      series={Graduate Texts in Mathematics},
   publisher={Springer-Verlag, New York},
        date={1995},
      volume={152},
        ISBN={0-387-94365-X},
         url={http://dx.doi.org/10.1007/978-1-4613-8431-1},
      review={\MR{1311028}},
}

\end{thebibliography}

\end{document}